\documentclass[a4paper,12pt]{amsart}
\usepackage[colorlinks=true,urlcolor=blue,
citecolor=red,linkcolor=blue,linktocpage,pdfpagelabels,
bookmarksnumbered,bookmarksopen]{hyperref}
\usepackage[active]{srcltx}
\usepackage{verbatim}
\usepackage{epsfig,graphicx,color,mathrsfs}
\usepackage{graphicx}
\usepackage{amsmath,amssymb,amsthm,amsfonts}
\usepackage{amssymb}
\usepackage[english]{babel}
\usepackage[scrtime]{prelim2e}
\usepackage[hyperpageref]{backref}
\usepackage[left=2.7cm,right=2.7cm,top=3cm,bottom=3cm]{geometry}

\newtheorem{thm}{Theorem}[section]

\newtheorem{lem}[thm]{Lemma}
\newtheorem{prop}[thm]{Proposition}
\newtheorem{rem}[thm]{Remark}
\newtheorem{defn}[thm]{Definition}

\newcommand{\R}{\mathbb R}

\newcommand{\ho}{H^\circ}
\newcommand{\n}{\nabla}
\newcommand{\la}{\langle}
\newcommand{\ra}{\rangle}

\newcommand{\di}{\,\text{\rmfamily\upshape d}}

\newcommand{\bh}{\mathcal B^H}
\newcommand{\bho}{\mathcal B^{H^\circ}}

\newcommand{\D}{\hbox{D}}

\begin{document}

\parindent 0pc
\parskip 6pt
\overfullrule=0pt

\title[Anisotropic Kelvin transform]
{Anisotropic Kelvin Transform}

\author[F. Esposito, G. Riey, B. Sciunzi, D. Vuono]{Francesco Esposito$^{*}$, Giuseppe Riey$^{*}$, Berardino Sciunzi$^{*}$, Domenico Vuono$^{*}$}
\address{$^{*}$Dipartimento di Matematica e Informatica, UNICAL, Ponte Pietro  Bucci 31B, 87036 Arcavacata di Rende, Cosenza, Italy}
\email{francesco.esposito, giuseppe.riey, domenico.vuono@unical.it}
\email{sciunzi@mat.unical.it}



\date{\today}

\thanks{The authors are members of INdAM.
    F. Esposito, G. Riey and B. Sciunzi are partially supported by PRIN project 2017JPCAPN (Italy):
    {\em Qualitative and quantitative aspects of nonlinear PDEs.}}

\date{\today}

\begin{abstract}
In this paper we introduce and study a new Kelvin-type transform in the
anisotropic setting. In particular, we deal with semilinear and quasilinear anisotropic elliptic problems in the entire space.
\end{abstract}

\maketitle

\section{Introduction}

The aim of this paper is to consider elliptic partial differential equations involving the Finsler Laplace operator:
$$
-\Delta^H u:= - \operatorname{div}(H(\nabla u) \nabla H(\nabla u))\,,
$$
where $H$ is a Finsler norm (see assumption $(h_H)$ below), and for $H(\xi)=|\xi|$, the operator $\Delta^H$ reduces to the classical Laplacian. Elliptic equations involving the anisotropic operators arise as Euler-Lagrange equations of Wulff-type energy functionals. 

In the first part of the present work, we shall consider  solutions to the equation
\begin{equation}\label{eq:AnisotropicEq}
	-\Delta^H u = f(x) \qquad \text{ in } \R^N,\quad N\geq 2\,,
\end{equation}
which is the Euler-Lagrange equation of the \emph{Wulff-type energy functional}
\begin{equation}\label{eq:Wulff}
	\mathcal{W}(u)=\int_{\R^N} \left[\frac {H(\nabla u)^2}{2}-f(x)u\right] \di x\,.
\end{equation}
As usual, a solution  has to be understood in the weak distributional meaning, i.e. $u \in H^{1}_{loc}(\R^N)$ is a weak solution of \eqref{eq:AnisotropicEq} if
\begin{equation}\label{eq debole}
	\int_{\R^N} H(\nabla u) \langle \nabla H(\nabla u), \nabla
	\varphi \rangle\,\di x\,=\,\int_{\R^N} f(x)\varphi\,\di x, \quad\forall \varphi\in
	C^1_c(\R^N)\,.
\end{equation}


We address here this deep issue and we succeed in providing a suitable Kelvin-type transformation in the Riemannian case, 
 \begin{equation}\tag{$h_M$}\label{eq:Hmatriciale}
 	H(\xi)=\sqrt{\langle M\xi,\xi\rangle} 
 \end{equation}
 for some symmetric and positive definite matrix $M$. In this setting, we define the anisotropic Kelvin-type transformation by
$$T_H: \R^N \setminus \{0\} \rightarrow \R^N \setminus \{0\}, \quad T_H(\xi):= \frac{\nabla H(\xi)}{H(\xi)},$$
and consequently 
$$\hat u:=\frac{u \circ T_H}{H^{N-2}}  \quad \text{in } \R^N \setminus \{0\},$$ 
see Section \ref{Sec3} for more details.

\begin{rem}\rm\label{genericFinsler}
	Although our main results will be proved in the Riemannian case \eqref{eq:Hmatriciale}, a part of our construction holds for a general Finsler norm. This is why many preliminary results will be stated in full generality and also we shall try to use the language of Finsler geometry, where it is possible. This will also highlight the geometric idea that is behind our construction.
\end{rem}

With such construction at hand, we prove the following.
\begin{thm}\label{teoKelvin} Let $H$ be given by \eqref{eq:Hmatriciale}. If $u\in H^1_{loc}(\R^N)$ is a locally bounded solution of \eqref{eq:AnisotropicEq},  with $f \in L^2_{loc}(\R^N)$, then $\hat u $ belongs to $D^{1,2}$ far from the origin, and weakly solves the dual equation
	\begin{equation}
		-\Delta^{H^\circ} \hat u = \frac{f \circ T_H}{H^{N+2}} \quad  \text{in } \R^N\setminus \{0\},
	\end{equation}
	where $H^\circ$ is the dual norm.
\end{thm}

\begin{rem}
Let us stress the fact that the result in Theorem \ref{teoKelvin} could be also obtained by using a suitable change of variable together with the techniques known for the standard Laplace operator. Such a proof would not be so easier or shorter. This led us to provide a proof that emphasizes the geometrical properties of the anisotropic Kelvin transformation.
\end{rem}

As observed by Ciraolo, Figalli and Roncoroni in \cite{CiFiRo}, a Kelvin-type approach doesn't work, in general, in the anisotropic context. This also happen in the Euclidean framework when $p \neq 2$. As we shall show adding an example in Remark \ref{counterexample}, a pure invariance of the equation is not expected for general Finsler norms, even in the semilinear case $p=2$. 

To the best of our knowledge this is the first result in this direction in the anisotropic framework. Let us also mention the work of Monti and Morbidelli \cite{MM}
for a non-trivial application of the Kelvin transformation for equations involving
the Grushin operators.

The last part of our paper is devoted to study the following quasilinear elliptic equation:
\begin{equation}\label{eq:Nlaplaciano}
	-\Delta^H_N \text{ } u = g(x) \qquad \text{in } \R^N, \quad N\geq 3\,,
\end{equation}

where 	$\Delta^H_N \text { }u:=\operatorname{div}(H^{N-1}(\nabla u)\nabla H(\nabla u)) $ is the Finsler $N$-Laplace operator, and $g\in L^\frac{N}{N-1}_{loc}(\R^N)$.

As pointed out above, we say that  $u \in W^{1, N}_{loc}(\R^N)$ is a weak solution of \eqref{eq:AnisotropicEq} if
\begin{equation}\label{eq debole N-Laplace}
	\int_{\R^N}  H^{N-1}(\nabla u)\langle \nabla H(\nabla u), \nabla
	\varphi \rangle\,\di x\,=\,\int_{\R^N} g(x)\varphi\,\di x, \quad\forall \varphi\in
	C^1_c(\R^N)\,.
\end{equation}

In the Euclidean framework, i.e. if $H(\xi)=|\xi|$, we refer to the seminal paper by W. Chen and C. Li \cite{ChenLi}, where the, making use of the Kelvin transformation, classified solutions to the Liouville equation ($g=e^u$) via an adaptation of the moving plane method. As remarked above, in quasilinear case the situation is much more involved and we refer to the work of P. Esposito \cite{Pierpaolo} for a classification result of the Liouville equation for $p=N>2$, with the use of the Kelvin transform, but with a different approach to the one carried out in \cite{ChenLi}.

\begin{thm}\label{eqNlaplace}
    Let $H$ be given be \eqref{eq:Hmatriciale}. If $u \in W^{1, N}_{loc}(\R^N)$ is a locally bounded solution of \eqref{eq:Nlaplaciano}, with $g\in L^\frac{N}{N-1}_{loc}(\R^N)$, then $u^*:=u\circ T_H$ weakly solves the dual equation
    \begin{equation}\label{Ndualequation}
    -\Delta ^{H^\circ}_N u^*= \frac{g\circ T_H}{H^{2N}} \quad \text{in } \R^N\setminus \{0\},
    \end{equation}
    where $H^{\circ}$ is the dual norm.
\end{thm}

We point out that under suitable assumptions on $f$ and $g$, due to the results in \cite{ACF, DB, Lib, Tolk}, it is possible to assume that solutions of \eqref{eq:AnisotropicEq} and \eqref{eq:Nlaplaciano} are of class $\mathcal{C}^{1,\alpha}$.

Concluding the introduction, let us mention that anisotropic operators arise in Finsler geometry, see \cite{BePa, CianSal, LePa, Sh1, Sh2}, as modelization e.g.~of material science \cite{CaHo, Gu}, relativity \cite{As}, biology \cite{AnInMa},
image processing \cite{EsOs, PeMa}. 

In the following we devote Section \ref{Sec1} to the introduction of the main geometrical notion used in the paper. In Section \ref{Sec3} we study the Kelvin-type transformation, defined above, and we prove our main results.

\section{Notation and preliminary results} \label{Sec1}

The aim of this section is to introduce some notation and to recall technical results  about anisotropic elliptic operator that will be involved in the proof of the main results.

\noindent In all the paper generic fixed and numerical constants will be denoted by $\mathcal{C}$ (with subscript in some cases) and they will vary from line to line.
For $a, b \in \R^N$ we denote by $a \otimes b$ the matrix whose entries are $(a \otimes b)_{ij}=a_ib_j$.
We remark that for any $v,w \in \R^N$ it holds that:
$$
\langle a \otimes b \ v, w \rangle = \langle b, v \rangle \langle a , w \rangle.
$$
Moreover, we denote with $I$ the identity matrix of order $N$.

Let $H$ be a function belonging to $C^2(\R^N \setminus \{0\})$.
$H$ is said a ``\emph{Finsler norm}'' if it satisfies the following assumptions:
\begin{itemize}
	\item[$(h_H)$]
\begin{enumerate}
    \item[(i)] $H(\xi)>0 \quad \forall \xi \in \R^N \setminus \{0\}$;

    \item[(ii)] $H(s \xi) = |s| H(\xi) \quad \forall \xi \in \R^N \setminus \{0\}, \, \forall s \in \R$;

    \item[(iii)] $H$ is \emph{uniformly elliptic}, that means the set $\mathcal{B}_1^H:=\{\xi \in \R^N  :  H(\xi) < 1\}$ is \emph{uniformly convex}
    \begin{equation}
        \exists \Lambda > 0: \quad \langle \D^2H(\xi)v, v \rangle \geq \Lambda |v|^2 \quad \forall \xi \in \partial \mathcal{B}_1^H, \; \forall v \in \nabla H(\xi)^\bot.
    \end{equation}
\end{enumerate}
\end{itemize}

Now we provide some remarks on Finsler norms.

\begin{rem}\label{RMK1}\rm
A set is said uniformly convex if the principal curvatures of its boundary are all strictly positive.
We point out that these kind of sets are also called ``\emph{strongly convex}'',
which is stronger a condition than the usual ``\emph{strict convexity}''.
Moreover, assumption (iii) is equivalent to assume that $\D^2 (H^2)$ is definite positive.
\end{rem}

\begin{rem}\rm
Actually, in literature the definition of ``\emph{Finsler norm}'' is not unique.
In differential geometry, it is usual to define this class of norms satisfying directly the assumption that $\D^2 (H^2)$ is definite positive. The last one allows to define the fundamental tensor, that is a very useful tool in order to show suitable properties of geometrical objects like curvatures, distances and so on.
However, in many framework, it is enough to require that $\mathcal{B}_1^H$ is only strictly convex, and not strongly convex. Therefore, it is usual to give a definition of Finsler norm without the requirement (iii) (see for instance \cite{BeFeKa, BePa, FeKa}),
but with the only assumption that $H^2$ has to be strictly convex.
For example, for $p\neq 2$,
the $p$-norm in $\R^N$ is not uniformly elliptic, because its unit ball $\mathcal{B}_1^H$ is not strongly convex
but only strictly convex.
For our computations we need the stronger condition (iii) and we choose
a particular class of Finsler norms satisfying it.
\end{rem}

Since $H$ a norm in $\R^N$, we immediately get that there exists $c_1,c_2>0$ such that:
\begin{equation}\label{H equiv euclidea}
c_1|\xi|\leq H(\xi)\leq c_2|\xi|,\,\quad\forall\, \xi\in\R^N.
\end{equation}

The \emph{dual norm} $\ho:\R^N\to [0,+\infty)$ is defined as:
$$
\ho(x)=\sup\{\la \xi,x\ra: H(\xi)\leq 1\}.
$$
It is easy to prove that $\ho$ is also a Finsler norm and it has the same regularity
properties of $H$.
Moreover it follows that $(\ho)^\circ=H$.
For $r>0$ and $\overline x\in\R^N$ we define:
$$
\bh_r(\overline x)=\{x\in\R^N: H(x-\overline x)\leq r\}
$$
and
$$
\bho_r(\overline x)=\{x\in\R^N: \ho(x-\overline x)\leq r\}.
$$
For simplicity, when $\overline x=0$, we set: $\bh_r=\bh_r(0),\,\bho_r=\bho_r(0)$.
In literature $\bh_r$ and $\bho_r$ are also called ``Wulff  shape'' and
``Frank diagram'' respectively.

We recall that, since $H$ is a differentiable and $1$-homogeneous function, it holds the Euler's characterization result, i.e.
\begin{equation}\label{eulero}
\la\n H(\xi),\xi\ra=H(\xi) \qquad\forall\,x\in\R^N,
\end{equation}
and
\begin{equation}\label{grad 0 omog}
\nabla H(t\xi)=\hbox{sign}(t)H(\xi) \qquad\forall\,\xi\in\R^N,\, \forall t\in\R,
\end{equation}
and the same is true for $H^\circ$.
Moreover, there holds following identities:
\begin{equation}\label{eq:PropFinsler1}
H(\nabla H ^\circ (x)) = 1 = H^\circ (\nabla H (x))
\end{equation}
and
\begin{equation}\label{eq:PropFinsler2}
    H(x)\nabla H^\circ(\nabla H(x))=x=H^\circ(x)\nabla H(\nabla H^\circ (x)),
\end{equation}
and we refer the reader to \cite{BePa} for a complete proof. For more details on Finsler geometry see for instance \cite{BaChSh, BePa, JaSa}.

As pointed out in Remark \ref{genericFinsler}, a part of our construction hold for generic Finsler norms, but the main results are proved in the Riemannian case \eqref{eq:Hmatriciale}. In this case
$$	H^\circ(\xi)=\sqrt{\langle M^{-1}\xi,\xi\rangle}.$$

We refer to \cite{CoFaVa2, CoFaVa1, FeKa} for a discussion on the geometry of the Riemannian case and for interesting related results.

\section{Anisotropic Kelvin transform} \label{Sec3}

The aim of this section is to prove Theorem \ref{teoKelvin}. In order to achieve our claim, we introduce the formal definition of Kelvin transformation in the context of Finsler geometry.

\begin{defn}
Let $H$ be a Finsler norm.
We define the map $T_H: \R^N \setminus \{0\} \rightarrow \R^N \setminus \{0\}$ as
\begin{equation}\label{def Kelvin}
T_H(\xi)= \frac{\nabla H(\xi)}{H(\xi)}.
\end{equation}
\end{defn}

Note that, when $H(\xi)=|\xi|$, $T_H$ coincides with the usual Euclidean Kelvin transform. Hence, we observe that it holds the following result.

\begin{prop}
The map $T_H$ is a global diffeomorphism
with inverse given by $T_{H^\circ}$.
\end{prop}

\begin{proof}
We show that, if $y=T_H(x)$, then $x=T_{H^\circ}(y)$.
By \eqref{grad 0 omog},\eqref{eq:PropFinsler1} and \eqref{eq:PropFinsler2} we have:
\begin{equation}\label{inversa kelvin}
	\begin{split}
 T_H(T_{H^\circ}(y))&= \n H\left(\frac{\n H^\circ(y)}{H^\circ(y)}\right) H^{-1} \left(\frac{\n H^\circ(y)}{H^\circ(y)}\right)\\
 &=\n H(\n H^\circ(y)) H^\circ(y) H^{-1}(\n H^\circ(y))\\
 &=H^\circ(y)\n H(\n H^\circ(y))) =y\,.
\end{split}
\end{equation}
\end{proof}
Now, we denote by $\D T_H$ the Jacobian matrix of $T_H$, namely:
$$
\D T_H=\frac{1}{H^2}\left[H\D^2 H-\n H\otimes\n H\right]
= \frac{1}{H^2}\left[\D\left(\frac 1 2\n (H^2)\right)-2\n H\otimes\n H\right]\,.
$$
In particular, if $H$ satisfies \eqref{eq:Hmatriciale}, we have
\begin{equation}\label{DT M}
\D T_H(x)=\frac{1}{H(x)^2}\left[M-\frac{2}{H(x)^2}Mx\otimes Mx\right]\,.
\end{equation}
Moreover, we set $J(x)= |\det \D T_H(x)|$. Using these notations, we are ready to prove the following key lemma.

\begin{lem}\label{determinantecostante}
If $H$ is a Finsler norm that satisfy \eqref{eq:Hmatriciale}, then
\begin{equation}\label{determinante}
    J(x)=\frac{\det {M}}{H(x)^{2N}} \quad \forall x\in \R^N \setminus \{0\}.
\end{equation}
\end{lem}

\begin{proof}
Let $L$ be a diagonal matrix with entries $d_1,\ldots,d_N$ and
denote by $\sqrt{L}$ the diagonal matrix whose entries are $\sqrt{d_1},\ldots,\sqrt{d_N}$.
Suppose $H(x)=\sqrt{\langle L x,x\rangle}$.
In view of \eqref{DT M} we have:
\begin{equation}
\begin{split}
\D T_H(x)&= \frac{1}{H(x)^2}\left(L-\frac{2}{\langle L x,x\rangle} L x \otimes L x \right)=\\
&=\frac{1}{H(x)^2}\sqrt{L}\left(I-\frac{2}{\langle \sqrt{L}x,\sqrt{L}x\rangle} \sqrt{L}x \otimes \sqrt{L}x \right)\sqrt{L}
\end{split}
\end{equation}
Since
\begin{equation}\label{MatriceRiey}
  \det \left(I-\frac{2}{|y|^2} y \otimes y \right)=1,\quad \forall y\in \R^N \setminus \{0\},
\end{equation}
we get \eqref{determinante} in the case $L$ diagonal.
Let $H$ be of the form  \eqref{eq:Hmatriciale}. The spectral theorem ensures that it holds the decomposition $C^tMC=L$, where $C$ and $D$ are respectively an orthogonal and diagonal matrices. Having in mind this fact, we have
\begin{equation}
    C^t\D T_H (x)C=
    \frac{1}{H(x)^2}\sqrt{L}
    \left(I-\frac{2}{\langle \sqrt{L}C^tx,\sqrt{L}C^tx\rangle} \sqrt{L}C^tx \otimes \sqrt{L}C^tx \right)
    \sqrt{L}.
\end{equation}
By (\ref{MatriceRiey}) we get the thesis.
\end{proof}

Assumption \eqref{eq:Hmatriciale} is crucial in the previous lemma, and this result is essential also in the proof of Theorem \ref{teoKelvin} and Theorem \ref{eqNlaplace}. Indeed, our approach doesn't work without Lemma \ref{determinantecostante}, as pointed out in the next remark.
\begin{rem}\rm \label{counterexample}
In the proof of Theorem \ref{teoKelvin} it will be crucial the fact that $H(x)^{2N}J(x)$  is constant, see \eqref{determinante}.
For a generical Finsler norm this property is not true.
For example, if we consider the norm 
$$H(x)=\sqrt[4]{x_1^4+3x_1^2x_2^2+x_2^4}$$
for $x=(x_1,x_2)\in\R^2$, then it follows that 
$H(x)^{2N}J(x)$ is not constant, even if $H$ satisfies all the conditions (i)-(ii)-(iii)
stated above in the definition of a Finsler norm.
\end{rem}

Now, in the same spirit of the seminal paper of Caffarelli, Gidas and Spruck \cite{CGS}, given a function $u(x)$, we define the action of the Kelvin transform on $u$ as follows
\begin{equation}\label{kelv}
\hat u(y):=\frac{1}{H(y)^{N-2}} u(T_H(y)),
\end{equation}
for every $y\in\R^N \setminus \{0\}$.
\begin{rem}\rm
It is easy to check that
\begin{equation}\label{u kelv inversa}
u(x):=\frac{1}{H^\circ(x)^{N-2}} \hat u(T_{H^\circ}(x)),
\end{equation}
for every $x \in \R^N \setminus \{0\}$.
In fact, using \eqref{eq:PropFinsler1}, by \eqref{kelv} we have:
\begin{equation}
\begin{split}
u(x)&=\hat u\left(T_{H^\circ}(x)\right)H\left(T_{H^\circ}(x)\right)^{N-2}=
  \hat u\left(T_{H^\circ}(x)\right)H\left(\frac{\n H^\circ(x)}{H^\circ(x)}\right)^{N-2}\\
&=\hat u\left(T_{H^\circ}(x)\right)\frac{1}{H^\circ(x)^{N-2}}H(\n H^\circ(x))^{N-2}=
\frac{\hat u\left(T_{H^\circ}(x)\right)}{H^\circ(x)^{N-2}}\,.
\end{split}
\end{equation}
\end{rem}

Now we are  ready to prove the first result of our paper.

\begin{proof}[Proof of Theorem \ref{teoKelvin}]
We use $\psi\in C^1_c(\R^N\setminus \{0\})$ as test function in \eqref{eq debole}.
Then $\varphi(x):=\psi(T_{H^\circ }(x))$ can also be used as test function in \eqref{eq debole}, obtaining:
\begin{equation}\label{eq:primaequazione}
	    \begin{split}
    \int_{\R^N} H(\nabla u(x)) \langle \nabla H(\nabla u(x)), \D T_{H^\circ}(x)\nabla \psi (T_{H^\circ}(x))\rangle\,&\di x\\
    =   \int_{\R^N} f(x) \psi (T_{H^\circ}(x))\,&\di x.
\end{split}
\end{equation}
Using the change of variable $y=T_{H^\circ}(x)$, \eqref{eq:primaequazione} becomes
\begin{equation}\label{eq:secondaquazione}
    \begin{split}
    \int_{\R^N} H(\nabla u(T_H(y))) \langle \D T_{H^\circ}(T_H(y)) \nabla H(\nabla u(T_H(y))), \nabla \psi (y)\rangle J(y)\,&\di y\,\\
    =\int_{\R^N} f(T_H(y))\psi (y)J(y)\,&\di y.
\end{split}
\end{equation}
Under our assumption \eqref{eq:Hmatriciale} we deduce that
\begin{equation}\label{eq:propietajacobianifinsler}
    \begin{split}
        H(\nabla u(T_H(y)))& \D T_{H^\circ}(T_H(y))\nabla H(\nabla u(T_H(y)))\\
        &  =H(y)^4H^{\circ}(\D T_H(y)\nabla u(T_H(y)))\nabla H^\circ(\D T_H(y)\nabla u(T_H(y))).
    \end{split}
\end{equation}
If we set $u^*(y):=u(T_H(y))$ and $f^*(y):=f(T_H(y))$,
by using \eqref{eq:propietajacobianifinsler} and Lemma \ref{determinantecostante},
\eqref{eq:secondaquazione} becomes:
\begin{equation}\label{eq:terzaequzione}
    \begin{split}
        \int_{\R^N} \frac{1}{H(y)^{2N-4}}H^{\circ}(\nabla u^*(y)) \langle \nabla H^{\circ}(\nabla u^*(y)), \nabla \psi (y) \rangle  \,&\di y\\
        =\int_{\R^N} \frac{1}{H(y)^{2N}}f^*(y)\psi (y) \,&\di y.
    \end{split}
\end{equation}
Let $\phi \in C^1_c(\R^N\setminus \{0\})$ be such that $\psi(y)=H(y)^{N-2}\phi(y)$.
Then $\psi$ can be used as test function in \eqref{eq:terzaequzione}, getting:
\begin{equation}
    \begin{split}
    &\int_{\R^N}\frac{1}{H(y)^{N-2}}H^{\circ}(\nabla u^*(y)) \langle \nabla H^{\circ}(\nabla u^*(y)), \nabla \phi (y) \rangle  \\
    &\quad -\int_{\R^N} H^{\circ}(\nabla u^*(y)) \langle \nabla H^{\circ}(\nabla u^*(y)),\nabla \left(\frac{1}{H(y)^{N-2}}\right)\rangle \phi(y) \,
    \di x\\
    & = \int_{\R^N} \frac{1}{H(y)^{N+2}}f^*(y)\phi (y) \, \di y
    \end{split}
\end{equation}
Recalling \eqref{grad 0 omog}, the fact that $H$ is a $1$-homogeneous function and that
$$
\nabla \hat u= \nabla \left(\frac{1}{H^{N-2}}\right) u^*+\frac{1}{H^{N-2}}\nabla u^*,
$$
we get
\begin{equation}
\begin{split}
&\int_{\R^N} H^{\circ}\left(\nabla \hat u(y)-\nabla \left(\frac{1}{H(y)^{N-2}}\right)u^*\right) \times \\
& \qquad \times \langle \nabla H^{\circ}\left(\nabla \hat u(y)-\nabla \left(\frac{1}{H(y)^{N-2}}\right)u^*\right), \nabla \phi (y) \rangle \,\di y\\
& \quad-\int_{\R^N} H^{\circ}(\nabla u^*(y)) \langle \nabla H^{\circ}(\nabla u^*(y)),\nabla \left(\frac{1}{H(y)^{N-2}}\right) \rangle \phi(y)  \,\di y\\
&=\int_{\R^N} \frac{1}{H(y)^{N+2}}f^*(y)\phi (y) \,\di y
\end{split}
\end{equation}
By \eqref{eq:Hmatriciale}, we deduce that $H^\circ \nabla H^\circ$ is linear and symmetric
with respect to the Euclidean inner product of $\R^N$, i.e.
$$
\langle H^\circ (x)\nabla H^\circ(x),y\rangle=\langle x,H^\circ (y) \nabla H^\circ (y)\rangle \quad \forall x,y\in \R^N.
$$
Hence, we have
\begin{equation}\label{eq:DualEq4}
\begin{split}
&\int_{\R^N} H^{\circ}(\nabla \hat u(y)) \langle \nabla H^{\circ}(\nabla \hat u(y)), \nabla \phi (y) \rangle  \,\di y\\
&\;-\int_{R^N} H^{\circ}\left(\nabla \left(\frac{1}{H(y)^{N-2}}\right)\right) \langle \nabla H^{\circ}\left(\nabla \left(\frac{1}{H(y)^{N-2}}\right) \right),\nabla (u^*(y)\phi (y))\rangle  \,\di y\,\\
&=\int_{R^N} \frac{1}{H(y)^{N+2}}f^*(y)\phi (y) \,\di y.
\end{split}
\end{equation}
Integrating by parts \eqref{eq:DualEq4}, we get
\begin{equation}
\begin{split}
&\int_{\R^N} H^{\circ}(\nabla \hat u(y)) \langle \nabla H^{\circ}(\nabla \hat u(y)), \nabla \phi (y) \rangle  \,\di y\\
&\quad +\int_{\R^N} \operatorname{div} \left( H^{\circ}\left(\nabla \left(\frac{1}{H(y)^{N-2}}\right)\right)\nabla H^{\circ}\left(\nabla \frac{1}{H(y)^{N-2}} \right) \right)u^*(y)\phi(y)  \,\di y\\
&=\int_{\R^N} \frac{1}{H(y)^{N+2}}f^*(y)\phi (y) \,\di y.
\end{split}
\end{equation}
Since $w(x):=H(x)^{2-N}$ is an $H^\circ$-harmonic function,
namely $\Delta^{H^\circ}w=0$, we get the thesis.

\end{proof}

To conclude, with the same spirit of the previous proof, we prove our last result.

\begin{proof}[Proof of Theorem \ref{eqNlaplace}]
Let us consider $\psi\in C^1_c(\R^N\setminus \{0\})$ and we set $\varphi:=\psi (T_{H^{\circ}})$. We note that $\varphi$ is a good test function and substituting in \eqref{eq:Nlaplaciano}, we get

\begin{equation}\label{eq:primaequazione3}
	    \begin{split}
    \int_{\R^N} H(\nabla u(x))^{N-1} \langle \nabla H(\nabla u(x)), \D T_{H^\circ}(x)\nabla \psi (T_{H^\circ}(x))\rangle\,&\di x\\
    =   \int_{\R^N} g(x) \psi (T_{H^\circ}(x))\,&\di x.
\end{split}
\end{equation}
We set $y=T_{H^\circ}(x)$, and using this variable change \eqref{eq:primaequazione3} becomes
\begin{equation}\label{eq:secondaquazione2}
    \begin{split}
    \int_{\R^N} H(\nabla u(T_H(y)))^{N-1} \langle \D T_{H^\circ}(T_H(y)) \nabla H(\nabla u(T_H(y))), \nabla \psi (y)\rangle J(y)\,&\di y\,\\
    =\int_{\R^N} g(T_H(y))\psi (y)J(y)\,&\di y.
\end{split}
\end{equation}
If we set $g^*(y):=g(T_H(y))$, and recalling 
\eqref{eq:propietajacobianifinsler} and Lemma \ref{determinantecostante},
\eqref{eq:secondaquazione2} becomes:
\begin{equation}\label{eq:terzaequzione3}
    \begin{split}
        \int_{\R^N} \frac{H^{N-2}(\nabla u(T_H(y)))}{H(y)^{2N-4}}H^{\circ}(\nabla u^*(y)) \langle \nabla H^{\circ}(\nabla u^*(y)), \nabla \psi (y) \rangle  \,&\di y\\
        =\int_{\R^N} \frac{g^*(y)\psi (y)}{H(y)^{2N}} \,&\di y.
    \end{split}
\end{equation}

Under assumption \eqref{eq:Hmatriciale}, it is easy to check 
\begin{equation}\label{proprietanormefinsler}
  H^{\circ}(DT_H(y)\xi)=\frac{H(\xi)}{H^2(y)} \quad \forall \xi \in \R^N.  
\end{equation}
By \eqref{proprietanormefinsler} we get the thesis.
\end{proof}

\end{document}